\documentclass[11pt,a4paper,twoside]{article}

\usepackage{amssymb,amsmath,graphicx,multicol,amsthm,epstopdf,fancyhdr,xcolor,titlesec}
\usepackage{mathrsfs,mathtools,amsfonts}
\usepackage[colorlinks=true, pdfstartview=FitV, linkcolor=blue, citecolor=blue, urlcolor=blue]{hyperref}
\usepackage{cleveref}
\usepackage[english]{babel}
\usepackage[bottom,multiple]{footmisc}
\usepackage[margin=2.5cm]{geometry}
\usepackage[font=scriptsize,labelfont=bf]{caption,subcaption}
\usepackage[nospace]{cite} 

\usepackage{lmodern}             

\usepackage{lastpage}

\titleformat{\section}[block]{\bfseries\large}{\thesection. }{2pt}{}
\titleformat{\subsection}[block]
  {\bfseries\itshape}  
  {\thesubsection.}{2pt}{}

\theoremstyle{plain}
\newtheorem{Th}{Theorem}[section]
\newtheorem{lemma}[Th]{Lemma}
\newtheorem{corollary}[Th]{Corollary}
\newtheorem{proposition}[Th]{Proposition}

\theoremstyle{definition}
\newtheorem{definition}[Th]{Definition}

\newtheorem{remark}[Th]{Remark}

\newtheorem{example}[Th]{Example}

\newtheorem{theorem}[Th]{Theorem}

\allowdisplaybreaks
\everymath{\displaystyle}

\newcommand{\R}{\mathbb{R}}

\newcommand{\changetoline}[1]{%
    \let\parallel\|%
    #1
}

\newcommand{\secref}[1]{%
  \hyperref[#1]{\textcolor{blue}{Section~\ref*{#1}}}%
}

\begin{document}
\thispagestyle{empty}
\begin{center}
{\Large \textbf{The Fractional Two-Sided Quaternionic Dunkl Transform and Heisenberg-Type Inequalities}}
\end{center}

\vspace{1cm}

\begin{center}
\textbf{Mohamed Essenhajy} \footnote{Corresponding Author: mohamed.essenhajy@edu.umi.ac.ma}\footnote{Faculty of Sciences, Moulay Ismail University of Meknes, Morocco.}
\end{center}

\leftskip=2cm

\vspace{2cm}

\noindent\textbf{Abstract}
\\\\
This report investigates the main definitions and fundamental properties of the fractional two-sided quaternionic Dunkl transform in two dimensions. We present key results concerning its structure and emphasize its connections to classical harmonic analysis. Special attention is given to inversion, boundedness, spectral behavior, and explicit formulas for structured functions such as radial or harmonic functions. Within this framework, we establish a generalized form of the classical Heisenberg-type uncertainty principle. Building on this foundation, we further extend the result by proving a higher-order Heisenberg-type inequality valid for arbitrary moments $p \geq 1$, with sharp constants characterized through generalized Hermite functions. Finally, by analyzing the interplay between the two-sided fractional quaternionic Dunkl transform and the two-sided fractional quaternionic Fourier transform, we derive a corresponding uncertainty principle for the latter.

\vspace{0.3cm}

\leftskip=0cm 

\vspace*{2cm}
\section{Introduction}
Harmonic analysis has evolved through a constant interplay between generalization and specialization, extending classical theories to encompass richer algebraic structures and operational frameworks. The fractional Fourier transform (FrFT), first introduced by Namias~\cite{namias1980}, epitomizes such a generalization by introducing a continuous order parameter that interpolates between the identity operator and the classical Fourier transform, with notable applications in quantum mechanics. Building on this foundation, a wide range of fractional transforms have subsequently been proposed in the literature~\cite{key-GB14,key-haouala,key-hleili,key-tyr,key-samad,key-zayed}. This fractional paradigm has also been integrated into more specialized settings, most prominently within Dunkl theory~\cite{key-GB14,key-GB16}, which extends Fourier analysis through reflection groups and differential-difference operators, as well as within quaternionic analysis through the two-sided fractional quaternion Fourier transform~\cite{key-LSQ21}. These fractional extensions combine continuous parameters with enriched algebraic features, offering powerful tools for the study of function spaces and spectral theory.  

Within this landscape, the fractional two-sided quaternionic Dunkl transform (FrQDT) arises as a natural synthesis that unifies three fundamental directions of generalization: the fractional order parameterization of the FrFT, the algebraic structure of quaternion-valued function spaces, and the reflection group symmetry inherent to Dunkl operators. This transform simultaneously extends the framework established by Ghazouani and Bouzeffour for fractional Dunkl operators~\cite{key-GB14,key-GB16} and the two-sided fractional quaternion Fourier transform introduced by Li et al.~\cite{key-LSQ21}.  

Uncertainty principles have long constituted a cornerstone of harmonic analysis, expressing intrinsic lower bounds on the simultaneous concentration of functions in both spatial and spectral domains, as demonstrated in various quaternionic frameworks~\cite{key-fahlaoui2023,key-uncertainty,key-fahlaoui2022}.  The classical Heisenberg inequality for the Fourier transform has inspired a systematic series of extensions to fractional, Dunkl, and quaternionic settings. In particular, uncertainty principles have recently been established for the fractional Dunkl transform~\cite{key-elgadiri1,key-GB16,key-ghosh,key-kallel} and for the fractional quaternionic Fourier transform~\cite{key-elgadiri2,key-elgadiri3}, highlighting the interplay between fractional parameters and algebraic structures in determining minimal uncertainty bounds. In this work, we address the case of transforms that incorporate all these features by proving a sharp Heisenberg-type inequality for the fractional two-sided quaternionic Dunkl transform. Our results provide the exact lower bound for the product of weighted spatial and spectral moments, thereby extending and unifying known inequalities in related frameworks, and further identify quaternionic Gaussian functions as the unique extremizers.  

The remainder of the paper is organized as follows. \secref{sec1} reviews the necessary preliminaries on one-dimensional fractional Dunkl theory and quaternion algebra. \secref{sec2} introduces the two-sided fractional quaternionic Dunkl transform and establishes its fundamental properties, including inversion formulas, a Plancherel theorem, and Bochner-type identities. \secref{sec3} presents our main result: the proof of the Heisenberg-type uncertainty principle together with the characterization of its extremal functions. 

\section{Preliminaries}\label{sec1}
\subsection{The one-dimensional fractional Dunkl analysis} 
In this subsection, we recall the essential elements of the one-dimensional fractional Dunkl analysis, including the definition of the fractional Dunkl operator, its associated kernel, and the corresponding transform. These notions extend the classical Dunkl framework by incorporating a deformation parameter \( \theta \in \mathbb{R} \), and they play a central role in the developments that follow. We write
\[
x \cdot y = \sum_{i=1}^{2} x_i y_i
\]
for the inner product on \( \mathbb{R}^2 \), and abbreviate \( x^2 = x \cdot x \). The Euclidean norm is defined by \( |x| = \sqrt{x \cdot x} \). For further details and proofs, we refer the reader to \cite{key-GB16}.

For a fixed parameter $\chi \geq 0$ and a function $f$ defined on $\mathbb{R}$, the fractional Dunkl operator $\Lambda_{\chi}^{\theta}$ introduced by Ghazouani and Bouzeffour \cite{key-GB16} is defined for all $f \in C^{1}(\mathbb{R})$ as below:
\[
\Lambda_{\chi}^{\theta} f(y) := \frac{d}{dy} f(y) + \frac{2\chi + 1}{y} \left[ \frac{f(y) - f(-y)}{2} \right] + i \cot(\theta) y f(y)
\]
where $\theta \in \mathbb{R} \setminus \pi\mathbb{Z}$.

This operator $\Lambda_{\chi}^\theta$ serves as the infinitesimal generator of the fractional Dunkl kernel \( E_{\chi,\theta}(x,y) \), which is characterized as the unique solution to the following eigenvalue problem:
\begin{align*}
\left\{
\begin{aligned}
&\Lambda^\theta_\chi f = \frac{i y}{\sin(\theta)} f \\
&f(0) = e^{-\frac{i}{2} y^2 \cot(\theta)}.
\end{aligned}
\right.
\end{align*}
The explicit expression of the solution to the above eigenvalue problem is given by
\[
E_{\chi,\theta}(x,y) = e^{-\frac{i}{2}(x^2 + y^2)\cot(\theta)}\, E_\chi \left(y,  \frac{i x}{\sin(\theta)} \right),
\]
where \( E_\chi(.,.) \) denotes the classical Dunkl kernel associated with a rank-one root system, that is, the reflection group \( \mathbb{Z}_2 \) acting on \( \mathbb{R} \). 

The function \( E_{\chi,\theta}(x,y) \) satisfies several important properties, which we summarize below.

\medskip

\begin{proposition}
Let \( \theta \in \mathbb{R} \setminus \pi \mathbb{Z} \). Then:
\begin{enumerate}
\item For all \( x \in \mathbb{R} \) and \( y \in \mathbb{C} \), the fractional Dunkl kernel \( E_{\chi,\theta}(x,y) \) admits the following integral representation:
\begin{equation}
E_{\chi,\theta}(x,y) = \frac{\Gamma(\chi+1)}{\sqrt{\pi}\,\Gamma\left(\chi+\frac{1}{2}\right)} 
e^{-\frac{i}{2}(x^2 + y^2)\cot(\theta)} 
\int_{-1}^1 e^{\frac{i x y t}{\sin(\theta)}} (1 - t^2)^{\chi - \frac{1}{2}} (1 + t)\,dt.
\end{equation}
\item In particular, for all \( x, y \in \mathbb{R} \), the kernel is uniformly bounded as follows:
\begin{equation}\label{BK1}
|E_{\chi,\theta}(x, y)| \leq 1.
\end{equation}
\item Moreover, there exists a constant \( K(\chi, \theta) > 0 \) such that for all \( x, y \in \mathbb{R} \), we have
\begin{equation}\label{BK2}
|E_{\chi,\theta}(x,y)| \leq K(\chi, \theta)\min\left(1, |x y|^{-(\chi+\frac{1}{2})}\right).
\end{equation}
\end{enumerate}
\end{proposition}
\medskip

Before proceeding further, we recall the weighted Lebesgue spaces naturally associated with the Dunkl setting, which will be used in the sequel.

\medskip
Throughout this paper, we denote by \( w_{\chi}(x) = |x|^{2\chi+1} \) the standard weight function in the Dunkl framework. For \( 1 \leq p \leq \infty \), we define the weighted Lebesgue space \( L^{p}_{\chi}(\mathbb{R}) \) as the set of all measurable functions \( f : \mathbb{R} \to \mathbb{C} \) such that
\[
\| f \|_{\chi,p} := 
\begin{cases}
\left( \displaystyle\int_{\mathbb{R}} |f(x)|^{p} w_{\chi}(x) \, dx \right)^{\frac{1}{p}}, & \text{if } 1 \leq p < \infty, \\[8pt]
\displaystyle\text{ess sup}_{x \in \mathbb{R}} |f(x)|, & \text{if } p = \infty,
\end{cases}
\]
is finite.

\medskip
We now introduce the definition of the fractional Dunkl transform, which will serve as a fundamental tool in the subsequent analysis.
\begin{definition}
Let \( n \in \mathbb{Z} \) and let \( f \in L^1_\chi(\mathbb{R}) \). The fractional Dunkl transform \( \mathcal{D}^\theta_\chi \) is defined as follows:
\begin{enumerate}
    \item If \( \theta = 2n\pi \), then \( \mathcal{D}^{\theta}_\chi f(y) = f(y) \).
    \item If \( \theta = (2n+1)\pi \), then \( \mathcal{D}^{\theta}_\chi f(y) = f(-y) \).
    \item If \( \theta \in \mathbb{R} \), then \( \mathcal{D}^{\theta + 2n\pi}_\chi f(y) = \mathcal{D}^\theta_\chi f(y) \).
    \item If \( \theta \in ](2n-1)\pi,\, (2n+1)\pi[ \), then
    \begin{equation}\label{DFT}
    \mathcal{D}^\theta_\chi f(x) = c^\theta_\chi \int_{\mathbb{R}} f(y)\, E_{\chi,\theta}(x,y)\, |y|^{2\chi+1} \,dy,
    \end{equation}
    where
    \[
    c^\theta_\chi = \frac{e^{i(\chi+1)(\hat{\theta}\pi/2 - \theta)}}{|\sin(\theta)|^{\chi+1}}\, \alpha_\chi, \qquad
    \alpha_\chi = \frac{1}{2^{\chi+1}\Gamma(\chi+1)}, \qquad
    \hat{\theta} := \mathrm{sgn}(\sin(\theta)).
    \]
\end{enumerate}
\end{definition}
This definition extends the classical Dunkl transform by incorporating a fractional parameter \( \theta \), providing a continuous interpolation between the identity and the Dunkl transform, and connecting it to other fundamental integral transforms in harmonic analysis.

\begin{remark}
If \( f \) is an even function, then the fractional Dunkl transform  coincides with the fractional Hankel transform \( H_\chi^\theta \) (with parameter \( \chi \)) given by \cite{kerr1991}:
\begin{equation}\label{HFT}
H_\chi^\theta f(x) = 2 c^\theta_\chi \int_0^{+\infty} e^{-\frac{i}{2}(x^2 + y^2) \cot \theta} \, j_\chi \left( \frac{x y}{\sin \theta} \right) f(y) y^{2\chi + 1} \, dy,
\end{equation}
where the normalized spherical Bessel function \( j_\chi \) is defined as
\[
j_\chi(x) := \Gamma(\chi + 1) \sum_{n=0}^{+\infty} \frac{(-1)^n (x/2)^{2n}}{n! \Gamma(n + \chi + 1)}.
\]
\end{remark}

The fractional Dunkl transform defined above possesses the following fundamental properties:
\begin{proposition}
Let $\theta \in \mathbb{R} \setminus \pi\mathbb{Z}$ and $f \in L^1_\chi(\mathbb{R})$. The fractional Dunkl transform $\mathcal{D}^\theta_\chi$ satisfies the following properties:
\begin{enumerate}
    \item \textbf{Continuity and boundedness:} $\mathcal{D}^\theta_\chi f$ belongs to $C_0(\mathbb{R})$ and
    \[
    \|\mathcal{D}^\theta_\chi f\|_{\chi, \infty} \leq \frac{1}{\Gamma(\chi + 1) (2|\sin(\theta)|)^{\chi+1}} \|f\|_{\chi, 1}.
    \]
    \item \textbf{Composition formula:} For all $\theta, \beta \in \mathbb{R}$ and $f \in L^1_\chi(\mathbb{R})$ such that $\mathcal{D}^\beta_\chi(f) \in L^1_\chi(\mathbb{R})$, we have
    \begin{equation}\label{CD}
    \mathcal{D}^\theta_\chi \circ \mathcal{D}^\beta_\chi(f) = \mathcal{D}^{\theta+\beta}_\chi(f),
    \end{equation}
   
    with equality almost everywhere when $\theta + \beta \in \pi \mathbb{Z}$.
    \item \textbf{Isometry on $L^2$:} If $f \in L^1_\chi(\mathbb{R}) \cap L^2_\chi(\mathbb{R})$, then $\mathcal{D}^\theta_\chi f \in L^2_\chi(\mathbb{R})$ and
    \[
    \|\mathcal{D}^\theta_\chi f\|_{2, \chi} = \|f\|_{2, \chi}.
    \]
    \item \textbf{Unitary extension:} The operator $\mathcal{D}^\theta_\chi$ admits a unique extension as a unitary operator on $L^2_\chi(\mathbb{R})$, with inverse given by 
\begin{equation}\label{UE}
(\mathcal{D}^\theta_\chi)^{-1} = \mathcal{D}^{-\theta}_\chi.
\end{equation}    
\end{enumerate}
\end{proposition}

The following theorem describes the spectral behavior of the fractional Dunkl transform.

\begin{theorem}\cite[Theorem~5.2]{key-GB16}\label{lager}
There exists a basis of the space $L^2_\chi(\mathbb{R})$ consisting of eigenfunctions of the fractional Dunkl transform $D_\chi^\theta$. Specifically, the family of generalized Hermite functions $\{h_n^\chi\}_{n=0}^\infty$ forms such a basis, and for each $n \in \mathbb{N}$, the following spectral relation holds:
\begin{equation}\label{SR}
D_\chi^\theta h_n^\chi(x) = e^{in\theta} h_n^\chi(x).
\end{equation}
\end{theorem}

\begin{theorem}[Bochner-type identity]\cite[Theorem 5.1]{key-GB14}\label{BIDT}
Let \( f \in L_\chi^1(\mathbb{R}) \cap L_\chi^2(\mathbb{R}) \) be of the form
\[
f(x) = p(x) \psi(|x|),
\]
where \( p \in \mathcal{H}_r^\chi \) and \( \psi: \mathbb{R}_+ \to \mathbb{R} \) is a radial function. Then
\begin{equation}\label{BIDFT}
D_\chi^\theta f(x) = e^{i r \theta} p(x) H_{r + \theta - \frac{1}{2}}^\theta \psi(|x|).
\end{equation}
In particular, if \( f \) is radial (i.e., \( p \) is constant), then
\begin{equation}\label{BIDFTR}
D_\chi^\theta f(x) = H_{r + \theta - \frac{1}{2}}^\theta \psi(|x|).
\end{equation}

Here,
\begin{itemize}
    \item \( \mathcal{H}_r^\chi = \ker \Delta_k \cap \mathcal{P}_r \) is the space of \( \chi \)-spherical harmonics of degree \( r \geq 0 \),
    \item \( \mathcal{P}_r \) is the space of homogeneous polynomials of total degree \( r \),
    \item \( \Delta_\chi \) denotes the Dunkl Laplacian,
    \item  \( H_\nu^\theta \) the fractional Hankel transform \eqref{HFT}.
\end{itemize}
\end{theorem}

\subsection{Quaternion algebra}
 The standard quaternion algebra $\mathbb{H}$ is a non-commutative algebra defined over $\mathbb{R}$ with three imaginary units $\textbf{\textit{i}}$, $\textbf{\textit{j}}$, and $\textbf{\textit{k}}$. The multiplication rules for these units are defined as follows:
\[
\textbf{\textit{i}}^2 = \textbf{\textit{j}}^2 = \textbf{\textit{k}}^2 = \textbf{\textit{ijk}} = -1,
\]
\[
\textbf{\textit{ij}} = \textbf{\textit{k}}, \quad \textbf{\textit{ji}} = -\textbf{\textit{k}},
\]
\[
\textbf{\textit{jk}} = \textbf{\textit{i}}, \quad \textbf{\textit{kj}} = -\textbf{\textit{i}},
\]
\[
\textbf{\textit{ki}} = \textbf{\textit{j}}, \quad \textbf{\textit{ik}} = -\textbf{\textit{j}}.
\]
Quaternions are  isomorphic to the Clifford algebra $Cl_{(0,2)}$ of $\R^{(0,2)}$:
 \begin{align}\label{iso}
 \mathbb{H} \cong Cl_{(0,2)}.
 \end{align}
 
Any quaternion $q \in \mathbb{H}$ can be written explicitly as:
\[
q = q_0 + q_1 \textbf{\textit{i}} + q_2 \textbf{\textit{j}} + q_3 \textbf{\textit{k}},  
\]
where $q_0, q_1, q_2, q_3 \in \mathbb{R}$.

The real part of $q$ is defined as: 
\[
\Re(q) := q_0.
\]
Likewise, the imaginary part is given by:
\[
\Im(q) := q_1 \textbf{\textit{i}} + q_2 \textbf{\textit{j}} + q_3 \textbf{\textit{k}}.
\]

The quaternion conjugation has the following properties, making it a linear anti-involution:
\[
\overline{\overline{q}} = q, \quad \overline{p+q} = \overline{p} + \overline{q}, \quad \overline{pq} = \overline{q} \, \overline{p}, \quad \forall p, q \in \mathbb{H}.
\]

This leads to the definition of the modulus of a quaternion $q \in \mathbb{H}$ as:
\[
\vert q \vert _{\scriptscriptstyle \mathbb{H}} := \sqrt{q\overline{q}} = \sqrt{q_{0}^{2} + q_{1}^{2} + q_{2}^{2} + q_{3}^{2}}, \quad \forall q \in \mathbb{H}.
\]
The modulus $\vert q \vert _{\scriptscriptstyle \mathbb{H}}$ reduces to the ordinary Euclidean modulus $\vert q \vert = \sqrt{q^{2}}$ when $q$ is a real number. Additionally, we observe that for $x \in \mathbb{R}^{2}$, $\vert x \vert _{\scriptscriptstyle \mathbb{H}} = \vert x \vert$.
Furthermore, for arbitrary $p, q \in \mathbb{H}$, the following identities hold:
\begin{align*}
\vert q^{2} \vert _{\scriptscriptstyle \mathbb{H}} = \vert q \vert^{2} _{\scriptscriptstyle \mathbb{H}}, ~~ \vert pq \vert _{\scriptscriptstyle \mathbb{H}} = \vert p \vert _{\scriptscriptstyle \mathbb{H}} \vert q \vert _{\scriptscriptstyle \mathbb{H}} ~~ \text{and} ~~ \vert p+q \vert _{\scriptscriptstyle \mathbb{H}} \leq \vert p \vert _{\scriptscriptstyle \mathbb{H}} + \vert q \vert _{\scriptscriptstyle \mathbb{H}}.
\end{align*}

The inverse of a non-zero quaternion is
\begin{align*}
q^{-1} = \frac{\overline{q}}{\vert q \vert^{2} _{\scriptscriptstyle \mathbb{H}}},
\end{align*}
which shows that $\mathbb{H}$ is a normed division algebra.

Due to (\ref{iso}), the inner product of two vectors $\textbf{\textit{x}}$ and $\textbf{\textit{y}}$ in $Cl_{(0,2)}$ is given by:
\begin{align}\label{nn11}
\left( \textbf{\textit{x}}, \textbf{\textit{y}} \right) := \sum_{i=1}^{2}x_{i}y_{i} = -\frac{1}{2}(\textbf{\textit{x}}\textbf{\textit{y}} + \textbf{\textit{y}}\textbf{\textit{x}}).
\end{align}
In particular, for a vector $\textbf{\textit{x}}$ in $Cl_{(0,2)}$, it follows that
\begin{align}\label{nn00}
\vert \textbf{\textit{x}} \vert_{\scriptscriptstyle \mathbb{H}}^{2} = -\textbf{\textit{x}}^{2}
\end{align}
In this paper, we consider quaternion-valued signals $f : \mathbb{R}^{2} \longrightarrow \mathbb{H}$. Such functions can be expressed as
\begin{align*}
f(x) = f_{0}(x) + \textbf{\textit{i}}f_{1}(x) + \textbf{\textit{j}}f_{2}(x) + \textbf{\textit{k}}f_{3}(x),
\end{align*}
where $f_{0}, f_{1}, f_{2}$, and $f_{3}$ are real-valued functions.

The inner product of two such functions is defined as:
\begin{align}\label{IQP}
\langle f; g \rangle = \int_{\mathbb{R}^2} f(x) \overline{g(x)} \, d\mu_{\chi_1, \chi_2}(x),
\end{align}
where the weighted measure is given by:
\[
d\mu_{\chi_1, \chi_2}(x) := |x_1|^{2\chi_1 + 1} |x_2|^{2\chi_2 + 1} \, dx_1 dx_2.
\]

In particular, for $f = g$, the norm is:
\[
\|f\|_{p, \scriptscriptstyle \mathbb{H}, \chi_1, \chi_2} := \left( \int_{\mathbb{R}^2} |f(x)|_{\scriptscriptstyle \mathbb{H}}^p \, d\mu_{\chi_1, \chi_2}(x) \right)^{\frac{1}{p}}, \quad \text{for } 1 \leq p < \infty,
\]
and
\[
\|f\|_{\infty, \scriptscriptstyle \mathbb{H}} := \operatorname{ess\,sup}_{x \in \mathbb{R}^2} |f(x)|_{\scriptscriptstyle \mathbb{H}}.
\]
Let $p \in [1, +\infty ]$, we denote by $L^{p}_{\chi_1, \chi_2}(\mathbb{R}^{2}; \mathbb{H})$ the weighted space of quaternion-valued functions $f$ measurable on $\mathbb{R}^{2}$ such that
\begin{align*}
\Vert f \Vert_{p,\scriptscriptstyle \mathbb{H}, \chi_1, \chi_2} < \infty \quad \text{if} \quad p < \infty,
\end{align*}
and 
\begin{align*}
\Vert f \Vert_{\infty,\scriptscriptstyle \mathbb{H}} < \infty \quad \text{if} \quad p = \infty.
\end{align*}
For a comprehensive introduction to quaternion algebras and their properties, see \cite{key-voight}.

\section{Definition and Fundamental Properties of the Fractional Quaternionic Dunkl Transform in Two Dimensions}\label{sec2}
In this section, we introduce and study the fractional quaternionic Dunkl transform in two dimensions, which extends the classical Dunkl framework to quaternion-valued functions and incorporates fractional parameters. We begin by presenting its precise definition, followed by a discussion of its structural properties, which will be instrumental in the subsequent analysis.
\begin{definition}
Let $\textbf{\textit{a}}$ and $\textbf{\textit{b}}$ be pure quaternions such that $\textbf{\textit{a}}^2 = \textbf{\textit{b}}^2 = -1$, and let
\[
f \in L^1_{\chi_1,\chi_2}(\mathbb{R}^2; \mathbb{H}).
\]
For each \( \theta_1, \theta_2 \in ((2n - 1)\pi,\,(2n + 1)\pi) \), with \( n \in \mathbb{Z} \), the \emph{fractional two-sided quaternionic Dunkl transform} (FrQDT) of \( f \) is defined by
\begin{equation}\label{FrQDT}
\mathcal{F}_{\chi_1, \chi_2}^{\theta_1, \theta_2;\, -\textbf{\textit{a}}, -\textbf{\textit{b}}}\{f\}(y)
=
c^{\theta_1}_{\chi_1} \, c^{\theta_2}_{\chi_2}
\int_{\mathbb{R}^2}
E^{\textbf{\textit{a}}}_{\chi_1,\theta_1}(x_1,y_1) \,
f(x_1, x_2) \,
E^{\textbf{\textit{b}}}_{\chi_2,\theta_2}(x_2,y_2)
\, d\mu_{\chi_1, \chi_2}(x),
\end{equation}
where the modified fractional Dunkl kernels are given by
\[
E^{\textbf{\textit{a}}}_{\chi_1,\theta_1}(x_1,y_1) = e^{-\frac{\textbf{\textit{a}}}{2}(x_1^2 + y_1^2) \cot\theta_1}
E_{\chi_1}\left(x_1, -\textbf{\textit{a}} \frac{y_1}{\sin\theta_1}\right),
\]
\[
E^{\textbf{\textit{b}}}_{\chi_2,\theta_2}(x_2,y_2) = 
E_{\chi_2}\left(x_2, -\textbf{\textit{b}} \frac{y_2}{\sin\theta_2}\right)
e^{-\frac{\textbf{\textit{b}}}{2}(x_2^2 + y_2^2) \cot\theta_2}.
\]

The constants \( c^{\theta_i}_{\chi_i} \), for \( i = 1,2 \), are defined by
\[
c^{\theta_i}_{\chi_i} = \frac{e^{\textbf{\textit{q}}_i(\chi_i+1)(\hat{\theta}_i \pi/2 - \theta_i)}}{|\sin(\theta_i)|^{\chi_i+1}} \, \alpha_{\chi_i}, 
\quad \text{where} \quad \alpha_{\chi_i} = \frac{1}{2^{\chi_i+1} \Gamma(\chi_i+1)},
\]
and \( \hat{\theta}_i := \operatorname{sgn}(\sin\theta_i) \). Here, \( \textbf{\textit{q}}_1 := \textbf{\textit{a}} \) and \( \textbf{\textit{q}}_2 := \textbf{\textit{b}} \).
\end{definition}

\begin{remark}[\textbf{Connections to Classical and Fractional Quaternionic Transforms}]\label{CCFQ}

We now discuss several  relations to other quaternionic transforms in the literature:

\begin{enumerate}
  \item \emph{Classical QDT at $\theta_1 = \theta_2 = \pi/2$.} When $\theta_1 = \theta_2 = \pi/2$, we have $\cot(\pi/2) = 0$ and $\sin(\pi/2) = 1$. Thus,
  \[
  \mathcal{F}_{\chi_1,\chi_2}^{\pi/2,\pi/2;\,-\textbf{\textit{a}},-\textbf{\textit{b}}}\{f\}(y)
  =
  c^{\pi/2}_{\chi_1} \, c^{\pi/2}_{\chi_2}
  \int_{\mathbb{R}^2}
    E_{\chi_1}(x_{1}, -\textbf{\textit{a}}\, y_{1}) \,
    f(x_{1},x_{2}) \,
    E_{\chi_2}(x_{2}, -\textbf{\textit{b}}\, y_{2}) \,
    d\mu_{\chi_1,\chi_2}(x),
  \]
  which reduces, with \( c^{\pi/2}_{\chi_1} = c^{\pi/2}_{\chi_2} = 1 \), to the classical two-sided quaternionic Dunkl transform~\cite{key-fah}.

  \item \emph{Reduction to FrQFT when $\chi_1 = \chi_2 = 0$.} If \( \chi_1 = \chi_2 = 0 \), then the Dunkl kernel reduces to the classical exponential kernel \( E_{0}(u,v) = e^{uv} \), and \( d\mu_{0,0}(x) \) is the standard Lebesgue measure. Taking \( \textbf{\textit{a}} = \textbf{\textit{i}} \), \( \textbf{\textit{b}} = \textbf{\textit{j}} \), we recover the two-sided fractional quaternionic Fourier transform (FrQFT) as introduced in~\cite{key-LSQ21}:
  \[
  \mathcal{F}_{0,0}^{\theta_1,\theta_2;\,-\textbf{\textit{i}},-\textbf{\textit{j}}}\{f\}(y)
  = 
  \int_{\mathbb{R}^2} 
    K_{\theta_1}(x_{1},y_{1})\,f(x_{1},x_{2})\,K_{\theta_2}(x_{2},y_{2})\, dx_1 dx_2.
  \]

  \item \emph{Connections to other integral transforms.} Several quaternionic integral transforms in the literature appear as special cases of the FrQFT, and since the FrQFT itself is contained in the more general FrQDT, these transforms can also be viewed as particular instances of the FrQDT framework.. For example:
  \begin{itemize}
    \item The classical quaternionic Fourier transform \cite{key-ell, key-777} corresponds to the FrQFT with \(\theta_1 = \theta_2 = \pi/2\).
    \item The two-dimensional quaternionic Mellin transform of fractional order (2D-QFrMT)  is related to the FrQFT via a logarithmic change of variables and exponential reweighting (see proposition 3.1 in~\cite{key-PG25}). Consequently, the 2D-QFrMT is also related to the FrQDT. This connection provides a unified framework to analyze scale and frequency simultaneously within the quaternionic setting.

  \end{itemize}
\end{enumerate}
\end{remark}
To better understand the internal structure of the fractional quaternionic Dunkl transform in two variables, it is useful to express it as a composition of one-dimensional fractional Dunkl transforms acting along separate coordinates. This leads to the following important factorization property:
\begin{remark}[\textbf{Tensor factorization}]\label{composition} 
The FrQDT can be viewed as a composition of two one-dimensional transforms with non-commutative ordering:
\begin{equation}\label{CE}
\mathcal{F}_{\chi_1,\chi_2}^{\theta_1,\theta_2; -\textbf{\textit{a}}, -\textbf{\textit{b}}} = \mathcal{D}_{\chi_2}^{\theta_2, \textbf{\textit{b}}} \circ \mathcal{D}_{\chi_1}^{\theta_1, \textbf{\textit{a}}},
\end{equation}
where:
\begin{itemize}
    \item $\mathcal{D}_{\chi_1}^{\theta_1, \textbf{\textit{a}}}$ acts on the first variable:
    \[
    (\mathcal{D}_{\chi_1}^{\theta_1, \textbf{\textit{a}}} f)(y_1, x_2) = c_{\chi_1}^{\theta_1} \int_{\mathbb{R}} E^{\textbf{\textit{a}}}_{\chi_1,\theta_1}(x_1,y_1) f(x_1, x_2) |x_1|^{2\chi_1+1} dx_1
    \]
    \item $\mathcal{D}_{\chi_2}^{\theta_2, \textbf{\textit{b}}}$ acts on the second variable of the intermediate function $g(y_1, x_2) := (\mathcal{D}_{\chi_1}^{\theta_1, \textbf{\textit{a}}} f)(y_1, x_2)$:
    \[
    (\mathcal{D}_{\chi_2}^{\theta_2, \textbf{\textit{b}}} g)(y_1, y_2) = c_{\chi_2}^{\theta_2} \int_{\mathbb{R}} g(y_1, x_2) E^{\textbf{\textit{b}}}_{\chi_2,\theta_2}(x_2,y_2) |x_2|^{2\chi_2+1} dx_2
    \]
\end{itemize}
The composition order is critical due to quaternion non-commutativity:
\[
\mathcal{D}_{\chi_2}^{\theta_2, \textbf{\textit{b}}} \circ \mathcal{D}_{\chi_1}^{\theta_1, \textbf{\textit{a}}} \neq \mathcal{D}_{\chi_1}^{\theta_1, \textbf{\textit{a}}} \circ \mathcal{D}_{\chi_2}^{\theta_2, \textbf{\textit{b}}}
\]
\end{remark}
We now turn to the main properties of the FrQDT, which play a central role in the subsequent analysis.
\begin{proposition}[\textbf{Boundedness}]
Let $f \in L^1_{\chi_1,\chi_2}(\mathbb{R}^2; \mathbb{H})$. The FrQDT satisfies
\[
\left\| \mathcal{F}_{\chi_1,\chi_2}^{\theta_1,\theta_2;\,-\textbf{\textit{a}},\,-\textbf{\textit{b}}} \{f\} \right\|_{\infty, \scriptscriptstyle \mathbb{H}} \leq C \| f \|_{1, \scriptscriptstyle \mathbb{H}, \chi_1, \chi_2},
\]
where
\[
C = \frac{1}{\Gamma(\chi_1 + 1) \Gamma(\chi_2 + 1) 2^{\chi_1 + \chi_2 + 2} |\sin(\theta_1)|^{\chi_1 + 1} |\sin(\theta_2)|^{\chi_2 + 1}}.
\]
\end{proposition}

\begin{proof}
By definition of the transform and properties of the quaternionic modulus,
\begin{align*}
\left| \mathcal{F}_{\chi_1,\chi_2}^{\theta_1,\theta_2;\,-\textbf{\textit{a}},\,-\textbf{\textit{b}}} \{f\}(y) \right|_{\scriptscriptstyle \mathbb{H}} 
&\leq |c^{\theta_1}_{\chi_1} c^{\theta_2}_{\chi_2}|_{\scriptscriptstyle \mathbb{H}} \int_{\mathbb{R}^2} \left| E^{\textbf{\textit{a}}}_{\chi_1,\theta_1}(x_1,y_1) f(x_1, x_2) E^{\textbf{\textit{b}}}_{\chi_2,\theta_2}(x_2,y_2) \right|_{\scriptscriptstyle \mathbb{H}} d\mu_{\chi_1,\chi_2}(x) \\
&\leq |c^{\theta_1}_{\chi_1} c^{\theta_2}_{\chi_2}|_{\scriptscriptstyle \mathbb{H}} \int_{\mathbb{R}^2} \left| E^{\textbf{\textit{a}}}_{\chi_1,\theta_1}(x_1,y_1) \right|_{\scriptscriptstyle \mathbb{H}}  |f(x_1, x_2)|_{\scriptscriptstyle \mathbb{H}}  \left| E^{\textbf{\textit{b}}}_{\chi_2,\theta_2}(x_2,y_2) \right|_{\scriptscriptstyle \mathbb{H}} d\mu_{\chi_1,\chi_2}(x).
\end{align*}
Using the uniform bound (\ref{BK1}), we have
\[
\left| \mathcal{F}_{\chi_1,\chi_2}^{\theta_1,\theta_2;\,-\textbf{\textit{a}},\,-\textbf{\textit{b}}} \{f\}(y) \right|_{\scriptscriptstyle \mathbb{H}} \leq |c^{\theta_1}_{\chi_1} c^{\theta_2}_{\chi_2}|_{\scriptscriptstyle \mathbb{H}} \int_{\mathbb{R}^2} |f(x_1, x_2)|_{\scriptscriptstyle \mathbb{H}} d\mu_{\chi_1,\chi_2}(x).
\]
Substituting the explicit constants
\[
|c^{\theta_i}_{\chi_i}|_{\scriptscriptstyle \mathbb{H}} = \frac{1}{|\sin(\theta_i)|^{\chi_i+1} 2^{\chi_i+1} \Gamma(\chi_i+1)},
\]
we obtain
\[
\left| \mathcal{F}_{\chi_1,\chi_2}^{\theta_1,\theta_2;\,-\textbf{\textit{a}},\,-\textbf{\textit{b}}} \{f\}(y) \right|_{\scriptscriptstyle \mathbb{H}} \leq C \| f \|_{1, \scriptscriptstyle \mathbb{H}, \chi_1, \chi_2}
\]
for all $y \in \mathbb{R}^2$. Taking the essential supremum over $y$ yields the result:
\[
\left\| \mathcal{F}_{\chi_1,\chi_2}^{\theta_1,\theta_2;\,-\textbf{\textit{a}},\,-\textbf{\textit{b}}} \{f\} \right\|_{\infty, \scriptscriptstyle \mathbb{H}} \leq C \| f \|_{1, \scriptscriptstyle \mathbb{H}, \chi_1, \chi_2}.
\]
\end{proof}
\begin{proposition}[\textbf{Composition Formula}]
Let $\theta_1, \theta_2, \beta_1, \beta_2 \in \mathbb{R}$ such that $\theta_j, \beta_j \notin \pi\mathbb{Z}$ for $j=1,2$. For any function $f \in L^1_{\chi_1,\chi_2}(\mathbb{R}^2; \mathbb{H})$ such that $\mathcal{F}_{\chi_1,\chi_2}^{\beta_1,\beta_2; -\textbf{\textit{a}}, -\textbf{\textit{b}}}\{f\} \in L^1_{\chi_1,\chi_2}(\mathbb{R}^2; \mathbb{H})$, the following composition formula holds:
\begin{equation}\label{CE3}
\mathcal{F}_{\chi_1,\chi_2}^{\theta_1,\theta_2; -\textbf{\textit{a}}, -\textbf{\textit{b}}} \circ \mathcal{F}_{\chi_1,\chi_2}^{\beta_1,\beta_2; -\textbf{\textit{a}}, -\textbf{\textit{b}}}\{f\}
= \mathcal{D}_{\chi_2}^{\theta_2 + \beta_2, \textbf{\textit{b}}} \circ \mathcal{D}_{\chi_1}^{\theta_1 + \beta_1, \textbf{\textit{a}}}(f).
\end{equation}
almost everywhere on $\mathbb{R}^2$, provided that $\theta_j + \beta_j \notin \pi\mathbb{Z}$ for $j=1,2$.
\end{proposition}

\begin{proof}
Let $f \in L^1_{\chi_1,\chi_2}(\mathbb{R}^2; \mathbb{H})$. By definition, the successive application of the two quaternionic fractional Dunkl transforms writes explicitly as:
\begin{align*}
&\left(\mathcal{F}_{\chi_1,\chi_2}^{\theta_1,\theta_2; -\textbf{\textit{a}}, -\textbf{\textit{b}}} \circ \mathcal{F}_{\chi_1,\chi_2}^{\beta_1,\beta_2; -\textbf{\textit{a}}, -\textbf{\textit{b}}}\right)\{f\}(y_1, y_2) \\
&= c^{\theta_1}_{\chi_1} c^{\theta_2}_{\chi_2} \int_{\mathbb{R}^2}
E^{\textbf{\textit{a}}}_{\chi_1,\theta_1}(x_1, y_1)
\left[
\mathcal{F}_{\chi_1,\chi_2}^{\beta_1,\beta_2; -\textbf{\textit{a}}, -\textbf{\textit{b}}}\{f\}(x_1, x_2)
\right]
E^{\textbf{\textit{b}}}_{\chi_2,\theta_2}(x_2, y_2)
\, d\mu_{\chi_1, \chi_2}(x).
\end{align*}
Substituting the inner transform and applying Fubini's theorem to interchange the integrals, we get
\begin{align*}
&\left(\mathcal{F}_{\chi_1,\chi_2}^{\theta_1,\theta_2; -\textbf{\textit{a}}, -\textbf{\textit{b}}} \circ \mathcal{F}_{\chi_1,\chi_2}^{\beta_1,\beta_2; -\textbf{\textit{a}}, -\textbf{\textit{b}}}\right)\{f\}(y_1, y_2) \\
&= c^{\theta_1}_{\chi_1} c^{\theta_2}_{\chi_2} c^{\beta_1}_{\chi_1} c^{\beta_2}_{\chi_2}
\int_{\mathbb{R}^2} \int_{\mathbb{R}^2}
E^{\textbf{\textit{a}}}_{\chi_1,\theta_1}(x_1, y_1)
E^{\textbf{\textit{a}}}_{\chi_1,\beta_1}(u_1, x_1)
f(u_1, u_2)
E^{\textbf{\textit{b}}}_{\chi_2,\beta_2}(u_2, x_2)
E^{\textbf{\textit{b}}}_{\chi_2,\theta_2}(x_2, y_2) \\
&\quad \times d\mu_{\chi_1, \chi_2}(u) d\mu_{\chi_1, \chi_2}(x).
\end{align*}
By the one-dimensional composition formula \eqref{CD} for the fractional Dunkl transform, we have
\[
\mathcal{D}_{\chi_j}^{\theta_j} \circ \mathcal{D}_{\chi_j}^{\beta_j} = \mathcal{D}_{\chi_j}^{\theta_j + \beta_j}, \quad j=1,2,
\]
which implies
\[
\left(\mathcal{F}_{\chi_1,\chi_2}^{\theta_1,\theta_2;\,-\textbf{\textit{a}},\,-\textbf{\textit{b}}} \circ \mathcal{F}_{\chi_1,\chi_2}^{\beta_1,\beta_2;\,-\textbf{\textit{a}},\,-\textbf{\textit{b}}}\right) = \mathcal{D}_{\chi_2}^{\theta_2 + \beta_2, \textbf{\textit{b}}} \circ \mathcal{D}_{\chi_1}^{\theta_1 + \beta_1, \textbf{\textit{a}}}.
\]
\end{proof}
This composition property directly leads to an inversion result when the involved parameters are taken as opposites, as stated below.

\begin{theorem}[\textbf{Inversion Formula}]
Let \( f \in L^2_{\chi_1,\chi_2}(\mathbb{R}^2; \mathbb{H}) \). The FrQDT  satisfies the inversion formula
\[
f(x) = \mathcal{F}_{\chi_1,\chi_2}^{-\theta_1,-\theta_2; \, -\textbf{\textit{a}}, -\textbf{\textit{b}}} \left\{ \mathcal{F}_{\chi_1,\chi_2}^{\theta_1,\theta_2; \, -\textbf{\textit{a}}, -\textbf{\textit{b}}} \{f\} \right\} (x),
\]
holding almost everywhere on \(\mathbb{R}^2\).
\end{theorem}

\begin{proof}
Applying the composition formula \eqref{CE3} and the unitarity property of the fractional Dunkl transform \eqref{UE}, we obtain:
\begin{align*}
\mathcal{F}_{\chi_1,\chi_2}^{-\theta_1,-\theta_2; \, -\textbf{\textit{a}}, -\textbf{\textit{b}}} \circ \mathcal{F}_{\chi_1,\chi_2}^{\theta_1,\theta_2; \, -\textbf{\textit{a}}, -\textbf{\textit{b}}}\{f\}
&= \mathcal{D}_{\chi_2}^{\theta_2 -\theta_2, \textbf{\textit{b}}} \circ \mathcal{D}_{\chi_1}^{\theta_1-\theta_1, \textbf{\textit{a}}}(f)\\
&= \mathrm{Id}(f) = f
\end{align*}
\end{proof}
Although the inversion formula provides a reconstruction result, the Plancherel theorem, stated below, is derived independently and characterizes the transform as a norm-preserving operator on $L^2_{\chi_1,\chi_2}(\mathbb{R}^2; \mathbb{H})$.

To facilitate the proof of the Plancherel identity, we begin by proving a preliminary result concerning the inner product structure under the transform.
\begin{lemma}\label{L1}
Let \(\mathcal{H}_{n,m}^{\chi_1,\chi_2}(x_1,x_2) := h_n^{\chi_1}(x_1) \, h_m^{\chi_2}(x_2)\) be the tensor product basis functions on \(\mathbb{R}^2\), where the orthonormal system \(\{h_n^\chi\}_{n \in \mathbb{N}}\) of generalized Hermite functions is defined as in Theorem~\ref{lager} . Then the family
\[
\left\{ \mathcal{H}_{n,m}^{\chi_1,\chi_2} \right\}_{n,m \in \mathbb{N}}
\]
forms a complete orthonormal basis of the weighted space \(L^2_{\chi_1,\chi_2}(\mathbb{R}^2; \mathbb{H})\) with respect to the inner product \eqref{IQP}. Moreover, each \(\mathcal{H}_{n,m}^{\chi_1,\chi_2}\) is an eigenfunction of the FrQDT with eigenvalues
\begin{equation}\label{SD}
\mathcal{F}_{\chi_1,\chi_2}^{\theta_1,\theta_2;\, -\textbf{\textit{a}}, -\textbf{\textit{b}}} \left\{ \mathcal{H}_{n,m}^{\chi_1,\chi_2} \right\} = e^{\textbf{\textit{a}} n \theta_1} e^{\textbf{\textit{b}} m \theta_2} \, \mathcal{H}_{n,m}^{\chi_1,\chi_2}.
\end{equation}
\end{lemma}

\begin{proof}
Let \(f \in L^2_{\chi_1,\chi_2}(\mathbb{R}^2; \mathbb{H})\). We decompose \(f\) into its real-valued components as
\[
f(x) = f_0(x) + \textbf{\textit{i}} f_1(x) + \textbf{\textit{j}} f_2(x) + \textbf{\textit{k}} f_3(x),
\]
where each \( f_\ell \in L^2_{\chi_1,\chi_2}(\mathbb{R}^2; \mathbb{R} ) \) for \( \ell = 0,1,2,3 \).

Since \( \mathcal{H}^{\chi_1,\chi_2}_{n,m} \) forms a  basis of the scalar space \( L^2_{\chi_1,\chi_2}(\mathbb{R}^2; \mathbb{R}) \), each component function \( f_\ell \) admits the following expansion:
\[
f_\ell(x_1,x_2) = \sum_{n=0}^\infty \sum_{m=0}^\infty c^{(\ell)}_{n,m} \, \mathcal{H}^{\chi_1,\chi_2}_{n,m}(x_1,x_2),
\]
where the Fourier–Hermite coefficients are given by
\[
c^{(\ell)}_{n,m} = \int_{\mathbb{R}^2} f_\ell(x_1,x_2) \, \mathcal{H}^{\chi_1,\chi_2}_{n,m}(x_1,x_2) \, d\mu_{\chi_1,\chi_2}(x),
\]
and satisfy \( c^{(\ell)}_{n,m} \in \mathbb{R} \).

By linearity, we then recover the quaternionic expansion of \( f \) as
\[
f(x) = \sum_{n=0}^\infty \sum_{m=0}^\infty \left( c^{(0)}_{n,m} + \textbf{\textit{i}}c^{(1)}_{n,m} + \textbf{\textit{j}}c^{(2)}_{n,m} + \textbf{\textit{k}}c^{(3)}_{n,m} \right) \, \mathcal{H}^{\chi_1,\chi_2}_{n,m}(x) =\sum_{n=0}^\infty \sum_{m=0}^\infty c_{n,m} \, \mathcal{H}^{\chi_1,\chi_2}_{n,m}(x),
\]
where the quaternion-valued coefficients \( c_{n,m} \in \mathbb{H} \) are defined by
\[
c_{n,m} := c^{(0)}_{n,m} + \textbf{\textit{i}}c^{(1)}_{n,m} + \textbf{\textit{j}}c^{(2)}_{n,m} + \textbf{\textit{k}}c^{(3)}_{n,m}.
\]
Then, the elementary tensors $\mathcal{H}_{n,m}^{\chi_1,\chi_2}(x_1,x_2) = h_n^{\chi_1}(x_1) h_m^{\chi_2}(x_2)$ form a basis of the space \(  L^2_{\chi_1,\chi_2}(\mathbb{R}^2; \mathbb{H}) \), and we compute:
\begin{align*}
\langle \mathcal{H}_{n,m}^{\chi_1,\chi_2}; \mathcal{H}_{n',m'}^{\chi_1,\chi_2} \rangle 
&= \int_{\mathbb{R}^2} h_n^{\chi_1}(x_1) h_m^{\chi_2}(x_2) \, \overline{h_{n'}^{\chi_1}(x_1) h_{m'}^{\chi_2}(x_2)} \, d\mu_{\chi_1,\chi_2}(x) \\
&= \left( \int_{\mathbb{R}} h_n^{\chi_1}(x_1) \overline{h_{n'}^{\chi_1}(x_1)} |x_1|^{2\chi_1+1} dx_1 \right)
\left( \int_{\mathbb{R}} h_m^{\chi_2}(x_2) \overline{h_{m'}^{\chi_2}(x_2)} |x_2|^{2\chi_2+1} dx_2 \right) \\
&= \delta_{n,n'} \delta_{m,m'}.
\end{align*}
where $\delta_{i,j}$ denotes the Kronecker delta, equal to $1$ if $i = j$ and $0$ otherwise.

It remains to compute the action of the FrQDT. Using the factorization formula \eqref{CE} and the spectral relation \eqref{SR}, we have for any fixed pair $(n,m) \in \mathbb{N}^2$,
\begin{align*}
\mathcal{F}_{\chi_1,\chi_2}^{\theta_1,\theta_2; \, -\textbf{\textit{a}}, -\textbf{\textit{b}}} \left\{ \mathcal{H}_{n,m}^{\chi_1,\chi_2} \right\}(x_1,x_2)
&= \left( \mathcal{D}_{\chi_2}^{\theta_2, \textbf{\textit{b}}} \circ \mathcal{D}_{\chi_1}^{\theta_1, \textbf{\textit{a}}} \right) \left( h_n^{\chi_1}(x_1) h_m^{\chi_2}(x_2) \right) \\
&= \mathcal{D}_{\chi_2}^{\theta_2, \textbf{\textit{b}}} \left( e^{\textbf{\textit{a}} n \theta_1} h_n^{\chi_1}(x_1) h_m^{\chi_2}(x_2) \right) \\
&= e^{\textbf{\textit{a}} n \theta_1}  e^{\textbf{\textit{b}} m \theta_2} h_n^{\chi_1}(x_1) h_m^{\chi_2}(x_2).
\end{align*}

Hence, each basis function \( \mathcal{H}_{n,m}^{\chi_1,\chi_2} \) is an eigenfunction of the FrQDT, with corresponding eigenvalue \( e^{\textbf{\textit{a}} n \theta_1}  e^{\textbf{\textit{b}} m \theta_2} \), which completes the proof.
\end{proof}
Based on the previous lemma, we now establish the Plancherel theorem for the FrQDT.
\begin{theorem}[\textbf{Plancherel Formula}]
Let \( f \in L^2_{\chi_1,\chi_2}(\mathbb{R}^2; \mathbb{H}) \). The quaternionic fractional Dunkl transform satisfies the identity
\[
\left\| \mathcal{F}_{\chi_1,\chi_2}^{\theta_1, \theta_2; \, -\textbf{\textit{a}}, -\textbf{\textit{b}}}\{f\} \right\|_{2, \scriptscriptstyle \mathbb{H}, \chi_1, \chi_2}^2
=
\left\| f \right\|_{2, \scriptscriptstyle \mathbb{H}, \chi_1, \chi_2}^2.
\]
\end{theorem}

\begin{proof}
By Lemma~\ref{L1}, the family \(\{ \mathcal{H}_{n,m}^{\chi_1,\chi_2} \}_{n,m \in \mathbb{N}}\) forms a complete orthonormal basis of \(L^2_{\chi_1,\chi_2}(\mathbb{R}^2; \mathbb{H})\) that diagonalizes the operator \(\mathcal{F}_{\chi_1,\chi_2}^{\theta_1,\theta_2; -\textbf{\textit{a}}, -\textbf{\textit{b}}}\), with corresponding eigenvalues \(e^{\textbf{\textit{a}} n \theta_1} e^{\textbf{\textit{b}} m \theta_2}\).

As a result, any \(f \in L^2_{\chi_1,\chi_2}(\mathbb{R}^2; \mathbb{H})\) admits the spectral expansion
\[
f = \sum_{n,m=0}^\infty c_{n,m} \, \mathcal{H}_{n,m}^{\chi_1,\chi_2},
\]
where the coefficients \(c_{n,m} \in \mathbb{H}\) satisfy
\[
\| f \|_{2,\scriptscriptstyle \mathbb{H}, \chi_1,\chi_2}^2 = \sum_{n,m=0}^\infty |c_{n,m}|_{\scriptscriptstyle \mathbb{H}}^2.
\]

Applying the fractional quaternionic Dunkl transform yields:
\[
\mathcal{F}_{\chi_1,\chi_2}^{\theta_1,\theta_2; -\textbf{\textit{a}}, -\textbf{\textit{b}}} \{ f \}(x)
= \sum_{n=0}^{\infty} \sum_{m=0}^{\infty} e^{\textbf{\textit{a}} n \theta_1} \, c_{n,m} \, e^{\textbf{\textit{b}} m \theta_2} \, \mathcal{H}^{\chi_1,\chi_2}_{n,m}(x).
\]

Hence,
\begin{align*}
\left\| \mathcal{F}_{\chi_1,\chi_2}^{\theta_1,\theta_2; -\textbf{\textit{a}}, -\textbf{\textit{b}}}\{f\} \right\|_{2, \scriptscriptstyle \mathbb{H}, \chi_1, \chi_2}^2
&= \sum_{n=0}^{\infty} \sum_{m=0}^{\infty} \left| e^{\textbf{\textit{a}} n \theta_1} \, c_{n,m} \, e^{\textbf{\textit{b}} m \theta_2} \right|_{\scriptscriptstyle \mathbb{H}}^2 \\
&= \sum_{n=0}^{\infty} \sum_{m=0}^{\infty} \left| c_{n,m} \right|_{\scriptscriptstyle \mathbb{H}}^2,
\end{align*}
where we used the fact that quaternionic exponentials are unimodular: \( |e^{\textbf{\textit{a}} n \theta_1}|_{\scriptscriptstyle \mathbb{H}} = |e^{\textbf{\textit{b}} m \theta_2}|_{\scriptscriptstyle \mathbb{H}} = 1 \).

Therefore,
\[
\left\| \mathcal{F}_{\chi_1,\chi_2}^{\theta_1,\theta_2; -\textbf{\textit{a}}, -\textbf{\textit{b}}} \{ f \} \right\|_{2, \scriptscriptstyle \mathbb{H}, \chi_1, \chi_2}^2
= \left\| f \right\|_{2, \scriptscriptstyle \mathbb{H}, \chi_1, \chi_2}^2.
\]

\end{proof}
In light of the established framework for the fractional quaternionic Dunkl transform, it is pertinent to examine its explicit action on well-structured classes of functions. Of particular interest are functions admitting a decomposition into products of radial components and quaternionic Dunkl harmonics, which play a fundamental role in this setting. The Bochner-type formula presented below furnishes a detailed characterization of the transform’s behavior on such functions. This formula elucidates the intricate interplay between radiality and harmonicity within the transform and constitutes a natural generalization of classical Bochner identities to the fractional quaternionic context.

\begin{theorem}[\textbf{Bochner-type formula}]\label{BT}
Let \( f \in L^1_{\chi_1,\chi_2}(\mathbb{R}^2;\mathbb{H}) \cap L^2_{\chi_1,\chi_2}(\mathbb{R}^2;\mathbb{H}) \) be a function of the form
\[
f(x_1, x_2) = p(x_1, x_2)\, \psi_1(|x_1|)\, \psi_2(|x_2|),
\]
where
\begin{itemize}
    \item \( p(x_1, x_2) \) is a quaternionic-valued polynomial, homogeneous of bidegree \( (r_1,r_2) \), which is \( \chi_1 \)-Dunkl harmonic in \( x_1 \) and \( \chi_2 \)-Dunkl harmonic in \( x_2 \);
    \item \( \psi_1, \psi_2 : \mathbb{R}_+ \to \mathbb{R} \) are real-valued radial functions.
\end{itemize}
Then, the FrQDT satisfies
\begin{equation}\label{BT1}
\mathcal{F}_{\chi_1, \chi_2}^{\theta_1, \theta_2;\, -\textbf{\textit{a}}, -\textbf{\textit{b}}}\{f\}
=
e^{\textbf{\textit{a}} r_1 \theta_1}\,  p(x_1, x_2) \,
H_{r_1 + \chi_1 - \frac{1}{2}}^{\theta_1} \psi_1(|x_1|)
\, e^{\textbf{\textit{b}} r_2 \theta_2}
\,H_{r_2 + \chi_2 - \frac{1}{2}}^{\theta_2} \psi_2(|x_2|),
\end{equation}
where, for \( i=1,2 \), \( H^{\theta_i}_\nu \) denotes the fractional Hankel transform of order \( \nu \) and angle \( \theta_i \).

In particular, if \( f \) is radial, then
\begin{equation}\label{BT2}
\mathcal{F}_{\chi_1, \chi_2}^{\theta_1,\theta_2;\, -\textbf{\textit{a}}, -\textbf{\textit{b}}}\{f\} =
H_{\chi_1 - \frac{1}{2}}^{\theta_1} \psi_1(|x_1|)\, H_{\chi_2 - \frac{1}{2}}^{\theta_2} \psi_2(|x_2|).
\end{equation}
\end{theorem}
\begin{proof}
Let \( f \in L^1_{\chi_1, \chi_2}(\mathbb{R}^2;\mathbb{H}) \cap L^2_{\chi_1, \chi_2}(\mathbb{R}^2;\mathbb{H}) \) be of the form
\[
f(x_1, x_2) = p(x_1, x_2)\, \psi_1(|x_1|)\, \psi_2(|x_2|),
\]
where \( p \) is a quaternionic-valued polynomial, homogeneous of bidegree \((r_1,r_2)\).

Applying the tensor factorization formula \eqref{CE}, we first perform the fractional Dunkl transform in the \( x_1 \)-variable, where the classical Bochner-type identity \eqref{BIDFT} applies:

\[
\mathcal{D}_{\chi_1}^{\theta_1, \textbf{\textit{a}}} f(x_1, x_2) 
= e^{\textbf{\textit{a}} r_1 \theta_1} p(x_1, x_2) H_{r_1 + \chi_1 - \frac{1}{2}}^{\theta_1} \psi_1(|x_1|) \psi_2(|x_2|).
\]

Since the fractional Hankel transform 
\( H_{r_1 + \chi_1 - \frac{1}{2}}^{\theta_1} \psi_1(|x_1|) \) 
is quaternion-valued and generally does not commute with the quaternionic polynomial \( p \), 
we decompose \( p \) into its real scalar components:
\[
p(x_1, x_2) = p_0(x_1, x_2) + \textbf{\textit{i}}\, p_1(x_1, x_2) + \textbf{\textit{j}}\, p_2(x_1, x_2) + \textbf{\textit{k}}\, p_3(x_1, x_2).
\]

We then apply the fractional Dunkl transform 
\( \mathcal{D}_{\chi_2}^{\theta_2, \textbf{\textit{b}}} \) 
linearly and component-wise on each scalar function 
\( p_\ell(x_1, x_2) \psi_2(|x_2|) \),
thus preserving the quaternionic structure without requiring commutation between 
\( p \) and 
\( H_{r_1 + \chi_1 - \frac{1}{2}}^{\theta_1} \psi_1(|x_1|) \).

Hence,
\begin{align*}
&p(x_1, x_2) H_{r_1 + \chi_1 - \frac{1}{2}}^{\theta_1} \psi_1(|x_1|) \psi_2(|x_2|) \\
&= p_0(x_1, x_2) H_{r_1 + \chi_1 - \frac{1}{2}}^{\theta_1} \psi_1(|x_1|) \psi_2(|x_2|) 
+ \textbf{\textit{i}}\, H_{r_1 + \chi_1 - \frac{1}{2}}^{\theta_1} \psi_1(|x_1|) p_1(x_1, x_2) \psi_2(|x_2|) \\
&+ \textbf{\textit{j}}\, H_{r_1 + \chi_1 - \frac{1}{2}}^{\theta_1} \psi_1(|x_1|) p_2(x_1, x_2) \psi_2(|x_2|) 
+ \textbf{\textit{k}}\, H_{r_1 + \chi_1 - \frac{1}{2}}^{\theta_1} \psi_1(|x_1|) p_3(x_1, x_2) \psi_2(|x_2|).
\end{align*}

Therefore,
\begin{align*}
&\mathcal{D}_{\chi_2}^{\theta_2, \textbf{\textit{b}}} \left[
p(x_1, x_2) H_{r_1 + \chi_1 - \frac{1}{2}}^{\theta_1} \psi_1(|x_1|) \psi_2(|x_2|)
\right] \\
&= \mathcal{D}_{\chi_2}^{\theta_2, \textbf{\textit{b}}} \left[ p_0(x_1, x_2) \psi_2(|x_2|) \right] H_{r_1 + \chi_1 - \frac{1}{2}}^{\theta_1} \psi_1(|x_1|) \\
&\quad + \textbf{\textit{i}}\, H_{r_1 + \chi_1 - \frac{1}{2}}^{\theta_1} \psi_1(|x_1|) \mathcal{D}_{\chi_2}^{\theta_2, \textbf{\textit{b}}} \left[ p_1(x_1, x_2) \psi_2(|x_2|) \right] \\
&\quad + \textbf{\textit{j}}\, H_{r_1 + \chi_1 - \frac{1}{2}}^{\theta_1} \psi_1(|x_1|) \mathcal{D}_{\chi_2}^{\theta_2, \textbf{\textit{b}}} \left[ p_2(x_1, x_2) \psi_2(|x_2|) \right] \\
&\quad + \textbf{\textit{k}}\, H_{r_1 + \chi_1 - \frac{1}{2}}^{\theta_1} \psi_1(|x_1|) \mathcal{D}_{\chi_2}^{\theta_2, \textbf{\textit{b}}} \left[ p_3(x_1, x_2) \psi_2(|x_2|) \right].
\end{align*}

Applying the fractional Dunkl transform \( \mathcal{D}_{\chi_2}^{\theta_2, \textbf{\textit{b}}} \) in the \(x_2\)-variable, the classical Bochner-type identity \eqref{BIDFT} yields, for each real-valued component \( p_\ell(x_1, x_2) \):
\[
\mathcal{D}_{\chi_2}^{\theta_2, \textbf{\textit{b}}} \left[ p_\ell(x_1, x_2) \psi_2(|x_2|) \right] 
= e^{\textbf{\textit{b}} r_2 \theta_2} p_\ell(x_1, x_2) H_{r_2 + \chi_2 - \frac{1}{2}}^{\theta_2} \psi_2(|x_2|).
\]

Replacing all components, we have
\begin{align*}
&\mathcal{D}_{\chi_2}^{\theta_2, \textbf{\textit{b}}} \left[
p(x_1, x_2) H_{r_1 + \chi_1 - \frac{1}{2}}^{\theta_1} \psi_1(|x_1|) \psi_2(|x_2|)
\right] \\
&= p_0(x_1, x_2) H_{r_1 + \chi_1 - \frac{1}{2}}^{\theta_1} \psi_1(|x_1|) e^{\textbf{\textit{b}} r_2 \theta_2} H_{r_2 + \chi_2 - \frac{1}{2}}^{\theta_2} \psi_2(|x_2|) \\
&\quad + \textbf{\textit{i}}\, p_1(x_1, x_2) H_{r_1 + \chi_1 - \frac{1}{2}}^{\theta_1} \psi_1(|x_1|) e^{\textbf{\textit{b}} r_2 \theta_2} H_{r_2 + \chi_2 - \frac{1}{2}}^{\theta_2} \psi_2(|x_2|) \\
&\quad + \textbf{\textit{j}}\, p_2(x_1, x_2) H_{r_1 + \chi_1 - \frac{1}{2}}^{\theta_1} \psi_1(|x_1|) e^{\textbf{\textit{b}} r_2 \theta_2} H_{r_2 + \chi_2 - \frac{1}{2}}^{\theta_2} \psi_2(|x_2|) \\
&\quad + \textbf{\textit{k}}\, p_3(x_1, x_2) H_{r_1 + \chi_1 - \frac{1}{2}}^{\theta_1} \psi_1(|x_1|) e^{\textbf{\textit{b}} r_2 \theta_2} H_{r_2 + \chi_2 - \frac{1}{2}}^{\theta_2} \psi_2(|x_2|).
\end{align*}

Factoring out \( H_{r_1 + \chi_1 - \frac{1}{2}}^{\theta_1} \psi_1(|x_1|)
\, e^{\textbf{\textit{b}} r_2 \theta_2} H_{r_2 + \chi_2 - \frac{1}{2}}^{\theta_2} \psi_2(|x_2|) \) from all terms yields the final expression

\[
\mathcal{F}_{\chi_1, \chi_2}^{\theta_1, \theta_2; \, -\textbf{\textit{a}}, -\textbf{\textit{b}}} \{ f \} (x_1, x_2)
= e^{\textbf{\textit{a}} r_1 \theta_1} p(x_1, x_2) H_{r_1 + \chi_1 - \frac{1}{2}}^{\theta_1} \psi_1(|x_1|)
\, e^{\textbf{\textit{b}} r_2 \theta_2} H_{r_2 + \chi_2 - \frac{1}{2}}^{\theta_2} \psi_2(|x_2|).
\]

Note that, due to quaternionic non-commutativity, the order of quaternionic exponentials relative to polynomial and transform factors is essential and must be preserved as above.

In particular, if \( f \) is radial, then a straightforward application of equation \eqref{BIDFTR} and similar reasoning yields the simplified expression.

\end{proof}

\begin{remark}
It is worth noting that the Bochner-type formula established above can also be extended to the case where \(p\) is a real-valued polynomial, while the functions \(\psi_1\) and \(\psi_2\) are quaternionic-valued radial functions. In this setting, the quaternionic structure is entirely carried by the radial functions, and the polynomial part remains commutative, thus simplifying the analysis.

However, due to the intrinsic non-commutativity of quaternions, it is not possible to simultaneously consider both \(p\) and the radial functions \(\psi_1, \psi_2\) as quaternionic-valued without imposing additional structural constraints. The interplay between quaternionic-valued polynomials and quaternionic-valued radial functions complicates the ordering of factors in the transform and prevents a straightforward generalization of the Bochner-type formula in that fully quaternionic setting.
\end{remark}

In order to apply the Bochner-type theorem to explicit radial examples, it is necessary to determine the action of the fractional Hankel transform on the Gaussian function. The following lemma provides this computation, forming the analytic backbone of the analysis that follows.
\begin{lemma}[\textbf{Fractional Hankel transform of a Gaussian}]\label{lem:frac-hankel-gaussian}
Let $\chi > -1$, $\alpha > 0$, and $0 < \theta < \pi$. Consider the fractional Hankel transform $H_\chi^\theta$ defined for suitable functions by
\begin{equation} \label{HFT}
H_\chi^\theta f(y) = 2c_\chi^\theta \int_0^{+\infty} e^{-\frac{i}{2}(x^2 + y^2) \cot \theta}\; j_\chi\left( \frac{xy}{\sin \theta} \right) f(x)\; x^{2\chi+1}\,dx,
\end{equation}
where $c_\chi^\theta$ is a normalization constant and $j_\chi$ is the normalized spherical Bessel function
\[
j_\chi(z) = \Gamma(\chi + 1) \sum_{n=0}^\infty \frac{(-1)^n (z/2)^{2n}}{n! \Gamma(n+\chi+1)}.
\]
Then, for any $\alpha > 0$, the Gaussian function $g(x) = e^{-\alpha x^2}$ satisfies
\begin{equation} \label{fractional-hankel-of-gaussian-reversed}
H_\chi^\theta \left[ e^{-\alpha x^2} \right] (y) =
c_\chi^\theta \, \Gamma(\chi + 1) \left( \alpha + \frac{i}{2}\cot \theta \right)^{-\chi - 1}
\exp\left( - \frac{i}{2} y^2 \cot \theta - \frac{y^2}{4 (\alpha + \frac{i}{2}\cot \theta)\sin^2 \theta} \right).
\end{equation}
\end{lemma}

\begin{proof}
Let $g(x) = e^{-\alpha x^2}$. Plugging into the reversed version of \eqref{HFT} yields
\[
H_\chi^\theta g(y) = 2c_\chi^\theta \int_0^{+\infty} e^{-\frac{i}{2}(x^2 + y^2)\cot\theta}\, j_\chi\left( \frac{xy}{\sin\theta} \right) e^{-\alpha x^2} x^{2\chi+1} dx.
\]
Group the exponentials:
\[
e^{-\frac{i}{2}(x^2 + y^2)\cot\theta} e^{-\alpha x^2}
= e^{-\frac{i}{2}y^2\cot\theta} e^{-x^2\left( \alpha + \frac{i}{2} \cot\theta \right)}.
\]
Set $A := \alpha + \frac{i}{2} \cot\theta$, so
\[
H_\chi^\theta g(y) = 2c_\chi^\theta\, e^{-\frac{i}{2}y^2 \cot\theta} 
\int_0^{+\infty} j_\chi\left( \frac{xy}{\sin\theta} \right) e^{-A x^2} x^{2\chi+1} dx.
\]
Expand $j_\chi$:
\[
j_\chi(z) = \Gamma(\chi+1) \sum_{n=0}^\infty \frac{(-1)^n (z/2)^{2n}}{n! \Gamma(n + \chi + 1)},
\]
so
\begin{align*}
&\int_0^{+\infty} j_\chi\left( \frac{xy}{\sin\theta} \right) e^{-A x^2} x^{2\chi+1} dx
= \Gamma(\chi+1) \sum_{n=0}^{\infty} \frac{(-1)^n (y/2\sin\theta)^{2n}}{n! \Gamma(n+\chi+1)}
   \int_0^{+\infty} x^{2n+2\chi+1} e^{-A x^2} dx .
\end{align*}
The inner integral is
\[
\int_0^\infty x^{2n + 2\chi + 1} e^{-A x^2} dx = \frac{1}{2} A^{-n - \chi - 1} \Gamma( n + \chi + 1 ).
\]
Thus, the sum gives
\begin{align*}
&\Gamma(\chi+1) \sum_{n=0}^{\infty} \frac{(-1)^n (y/2\sin\theta)^{2n}}{n! \Gamma(n+\chi+1)} \frac{1}{2} A^{-n - \chi - 1} \Gamma( n + \chi + 1 ) \\
&= \frac{1}{2} \Gamma(\chi+1) A^{-\chi-1} \sum_{n=0}^{\infty} \frac{1}{n!} \left( - \frac{y^2}{4A\sin^2 \theta} \right)^n \\
&= \frac{1}{2} \Gamma(\chi+1) A^{-\chi-1} \exp\left( -\frac{y^2}{4A\sin^2 \theta} \right ).
\end{align*}
Therefore,
\[
H_\chi^\theta g(y) = 2 c_\chi^\theta\, e^{-\frac{i}{2} y^2 \cot\theta } \left[ \frac{1}{2} \Gamma(\chi+1) A^{-\chi-1} \exp\left( -\frac{y^2}{4A\sin^2\theta} \right ) \right]
\]
\[
= c_\chi^\theta\, \Gamma(\chi+1) \left( \alpha + \frac{i}{2} \cot\theta \right )^{-\chi-1}
\exp\left( -\frac{i}{2} y^2\cot\theta 
- \frac{y^2}{4 (\alpha+ \frac{i}{2}\cot\theta ) \sin^2\theta }\right ).
\]

\end{proof}
Equipped with the expression for the image of the Gaussian under the fractional Hankel transform, we can now state an explicit corollary describing the action of the fractional quaternionic Dunkl transform on the Gaussian.

\begin{corollary}
Let \( f : \mathbb{R}^2 \to \mathbb{H} \) be the quaternion-valued function defined by
\[
f(x) = e^{-\alpha |x|^2},
\]
where \( \alpha > 0 \). Then, the image of \( f \) under the FrQDT is given explicitly by
\begin{align*}
& \mathcal{F}_{\chi_1, \chi_2}^{\theta_1, \theta_2; \, -\textbf{\textit{a}}, -\textbf{\textit{b}}} \{ f \}(y_1, y_2) \\
&= \left[
c_{\chi_1 - \frac{1}{2}}^{\theta_1} \, \Gamma\left( \chi_1 + \frac{1}{2} \right) 
\left( \alpha + \frac{\textbf{\textit{a}}}{2} \cot \theta_1 \right)^{ -\chi_1 - \frac{1}{2} } 
\exp\left( -\frac{\textbf{\textit{a}}}{2} y_1^2 \cot \theta_1 - \frac{y_1^2}{4 \left( \alpha + \frac{\textbf{\textit{a}}}{2} \cot \theta_1 \right) \sin^2 \theta_1} \right)
\right] \\
&\quad \times \left[
c_{\chi_2 - \frac{1}{2}}^{\theta_2} \, \Gamma\left( \chi_2 + \frac{1}{2} \right) 
\left( \alpha + \frac{\textbf{\textit{b}}}{2} \cot \theta_2 \right)^{ -\chi_2 - \frac{1}{2} } 
\exp\left( -\frac{\textbf{\textit{b}}}{2} y_2^2 \cot \theta_2 - \frac{y_2^2}{4 \left( \alpha + \frac{\textbf{\textit{b}}}{2} \cot \theta_2 \right) \sin^2 \theta_2} \right)
\right].
\end{align*}
\end{corollary}

\begin{proof}
The result follows directly by applying identity \eqref{BIDFTR} together with Lemma~\ref{lem:frac-hankel-gaussian}.
\end{proof}
The following example illustrates a concrete instance of a quaternionic function which is an eigenfunction of the fractional quaternionic Dunkl transform defined previously. This further highlights the spectral properties of the transform on Gaussian-type signals with polynomial weights.

\begin{example}[\textbf{Quaternionic Gaussian-Type Eigenfunction of the FrQDT}]
Consider the quaternion-valued function
\[
f(x) = (t_1 + r_1 \textbf{\textit{a}})\, x_1^{\chi_1} e^{-\frac{x_1^2}{2}} \, x_2^{\chi_2} e^{-\frac{x_2^2}{2}} \, (t_2 + r_2 \textbf{\textit{b}}),
\]
where $x=(x_1,x_2) \in \mathbb{R}^2$ and $t_j, r_j \in \mathbb{R}$ are scalars.

Then, $f$ is an eigenfunction of the FrQDT corresponding to the eigenvalue $1$. More precisely, for all $y = (y_1,y_2) \in \mathbb{R}^2$, we have
\[
\mathcal{F}_{\chi_1, \chi_2}^{\theta_1, \theta_2; \, -\textbf{\textit{a}}, -\textbf{\textit{b}}} \{ f \}(y) = (t_1 + r_1 \textbf{\textit{a}})\, y_1^{\chi_1} e^{-\frac{y_1^2}{2}} \, y_2^{\chi_2} e^{-\frac{y_2^2}{2}} \, (t_2 + r_2 \textbf{\textit{b}}).
\]

Indeed, we recall that (as established in Theorem 3.3 of \cite{kerr1991}) for an integer $n \geq 0$, the function
\[
\varphi_n^{(\chi)}(x) = x^{\chi + \frac{1}{2}} e^{-\frac{x^2}{2}} L_n^{(\chi)}(x^2)
\]
is an eigenfunction of the fractional Hankel transform $H_\chi^\theta$ with eigenvalue $e^{i n \theta}$.

For the particular case $n=0$, since $L_0^{(\chi)}(z) = 1$ for all $z$, the function becomes
\[
\varphi_0^{(\chi)}(x) = x^{\chi + \frac{1}{2}} e^{-\frac{x^2}{2}}.
\]

Applying the fractional Hankel transform to this function yields
\[
H_\chi^\theta \left( x^{\chi + \frac{1}{2}} e^{-\frac{x^2}{2}} \right)(y) = y^{\chi + \frac{1}{2}} e^{-\frac{y^2}{2}}.
\]

Define the quaternionic constants
\[
Q_L = t_1 + r_1 \textbf{\textit{a}}, \quad Q_R = t_2 + r_2 \textbf{\textit{b}},
\]
and set the scalar radial functions
\[
g_1(x_1) = x_1^{\chi_1} e^{-\frac{x_1^2}{2}}, \quad g_2(x_2) = x_2^{\chi_2} e^{-\frac{x_2^2}{2}}.
\]

Then, \(f\) can be expressed as the tensor product
\[
f(x_1, x_2) = Q_L \, g_1(x_1) \, g_2(x_2) \, Q_R.
\]

Since \(Q_L\) and \(Q_R\) are quaternionic constants independent of the integration variables, they factor out of the fractional Hankel transforms \(H_{\chi_1 - \frac{1}{2}}^{\theta_1}\) and \(H_{\chi_2 - \frac{1}{2}}^{\theta_2}\). By linearity and preservation of multiplication order, we have
\[
\begin{aligned}
\mathcal{F}_{\chi_1, \chi_2}^{\theta_1, \theta_2; \, -\textbf{\textit{a}}, -\textbf{\textit{b}}} \{ f \}(y_1, y_2)
&= Q_L \left( H_{\chi_1 - \frac{1}{2}}^{\theta_1} g_1 \right)(y_1) \left( H_{\chi_2 - \frac{1}{2}}^{\theta_2} g_2 \right)(y_2) Q_R \\
&= (t_1 + r_1 \textbf{\textit{a}})\, y_1^{\chi_1} e^{-\frac{y_1^2}{2}} \, y_2^{\chi_2} e^{-\frac{y_2^2}{2}} \, (t_2 + r_2 \textbf{\textit{b}}).
\end{aligned}
\]
\end{example}

\section{Heisenberg-type inequality for the fractional two-sided quaternionic Dunkl transform}\label{sec3}
In this section, we establish a Heisenberg-type uncertainty principle for the FrQDT defined on weighted quaternion-valued \(L^{2}\)-spaces. Extending classical uncertainty results to this non-commutative, fractional, and Dunkl-framework setting, we derive a sharp inequality bounding the product of weighted variances of a function and its transform. We also characterize the equality case, showing that quaternionic Gaussian functions uniquely attain the bound.

\begin{theorem}[\textbf{Heisenberg-type inequality}]\label{thm:Heis-FrQDT}
Let  \( f\in L^{2}_{\chi_{1},\chi_{2}}(\mathbb{R}^{2};\mathbb{H}) \). The fractional two-sided quaternionic Dunkl transform satisfies

\begin{eqnarray}\label{eq:Heis-product}
\left( \int_{\mathbb{R}^2} |x|^2\, |f(x)|_{\scriptscriptstyle \mathbb{H}}^{\,2}\, d\mu_{\chi_1,\chi_2}(x) \right)
\left( \int_{\mathbb{R}^2} |y|^2\, \big|\mathcal{F}_{\chi_1,\chi_2}^{\theta_1,\theta_2; \,-\textbf{\textit{a}},-\textbf{\textit{b}}}\{f\}(y)\big|_{\scriptscriptstyle \mathbb{H}}^{\,2}\, d\mu_{\chi_1,\chi_2}(y) \right)
\\
\;\ge\;
\big( (2\chi_1+1)+(2\chi_2+1) \big)^2
\left( \int_{\mathbb{R}^2} |f(x)|_{\scriptscriptstyle \mathbb{H}}^{\,2}\, d\mu_{\chi_1,\chi_2}(x) \right)^{\!2}.\nonumber
\end{eqnarray}

Equality holds if and only if \( f \) is a quaternionic Gaussian of the form
\[
f(x_1, x_2) =  C e^{-\frac{| x |^2}{2}},
\]
where  \( C \in \mathbb{H} \)  is a fixed quaternionic scalar constant.
\end{theorem}

\begin{proof}
Consider the one-dimensional orthonormal basis \(\{ h_n^{\chi}(x) \}_{n \ge 0}\) in \(L^{2}_{\chi}(\mathbb{R})\). The generalized Hermite functions satisfy the recurrence relation
\[
x^2 h_n^{\chi}(x)
= \alpha_{n+1}^{(\chi)}\, h_{n+2}^{\chi}(x) 
+ \beta_n^{(\chi)}\, h_n^{\chi}(x) 
+ \alpha_n^{(\chi)}\, h_{n-2}^{\chi}(x),
\]
where the coefficients \(\alpha_n^{(\chi)}\) and \(\beta_n^{(\chi)}\) depend linearly on \(n\) and \(\chi\),  with  \( \alpha_n^{(\chi)} = 0 \text{ for } n < 0\). Explicitly, the diagonal coefficients satisfy the relation
\[
\beta_n^{(\chi)} = 2n + 2\chi + 1,
\]
while \(\alpha_n^{(\chi)} > 0\) for large enough \(n\).

In two dimensions, the tensor product basis reads
\[
\mathcal{H}_{n,m}^{\chi_1,\chi_2}(x_1,x_2) = h_n^{\chi_1}(x_1) \, h_m^{\chi_2}(x_2).
\]
The multiplication operator by \(|x|^2 = x_1^2 + x_2^2\) acts via
\[
|x|^2 \mathcal{H}_{n,m}^{\chi_1,\chi_2}
= (x_1^2 h_n^{\chi_1}) h_m^{\chi_2} + h_n^{\chi_1} (x_2^2 h_m^{\chi_2}),
\]
and employing the recurrence relation for each factor yields
\begin{align}\label{HRFractional}
|x|^2 \mathcal{H}_{n,m}^{\chi_1,\chi_2}(x_1,x_2)
&= \alpha_{n+1}^{(\chi_1)} \mathcal{H}_{n+2,m}^{\chi_1,\chi_2}(x_1,x_2)
+ \alpha_n^{(\chi_1)} \mathcal{H}_{n-2,m}^{\chi_1,\chi_2}(x_1,x_2) \notag \\
&\quad + \alpha_{m+1}^{(\chi_2)} \mathcal{H}_{n,m+2}^{\chi_1,\chi_2}(x_1,x_2)
+ \alpha_m^{(\chi_2)} \mathcal{H}_{n,m-2}^{\chi_1,\chi_2}(x_1,x_2) \notag \\
&\quad + A_{n,m} \mathcal{H}_{n,m}^{\chi_1,\chi_2}(x_1,x_2),
\end{align}
where the diagonal coefficient is
\[
A_{n,m} := \beta_n^{(\chi_1)} + \beta_m^{(\chi_2)} = 2(n+m) + 2(\chi_1 + \chi_2) + 2.
\]
The minimal value of \(A_{n,m}\) occurs at \((n,m) = (0,0)\):
\[
A_{\min} = 2(\chi_1 + \chi_2) + 2 > 0,
\]
so we write
\[
A_{n,m} = A_{\min} + B_{n,m}, \quad B_{n,m} := 2(n+m) \ge 0.
\]
By Lemma \ref{L1}, the tensorized Hermite family \(\{\mathcal H_{n,m}^{\chi_1,\chi_2}\}_{n,m\in \mathbb{N}}\) forms a complete orthonormal basis of \(L^2_{\chi_1,\chi_2}(\mathbb{R}^2; \mathbb{H})\), allowing the expansion
\[
f = \sum_{n,m \ge 0} c_{n,m}\, \mathcal H_{n,m}^{\chi_1,\chi_2}, \quad c_{n,m} \in \mathbb{H}.
\]
Using \eqref{HRFractional} and the orthogonality, we obtain
\begin{align}\label{eq:X}
\| |x| f \|_{2,\scriptscriptstyle \mathbb{H},\chi_1,\chi_2}^2 = \int_{\mathbb{R}^2} |x|^2 |f(x)|_{\scriptscriptstyle \mathbb{H}}^2\, d\mu_{\chi_1,\chi_2}(x)
= \sum_{n,m \ge 0} A_{n,m} |c_{n,m}|_{\scriptscriptstyle \mathbb{H}}^2.
\end{align}
Similarly, employing the eigenfunction relations \eqref{SD} and the unimodularity of quaternionic exponentials, an analogous formula holds for the transform domain:
\begin{align}\label{eq:Y}
\| |y| \mathcal{F}_{\chi_1, \chi_2}^{\theta_1, \theta_2;\, -\textbf{\textit{a}}, -\textbf{\textit{b}}}\{f\}(y) \|_{2,\scriptscriptstyle \mathbb{H},\chi_1,\chi_2}^2 =\int_{\mathbb{R}^2} |y|^2 \big|\mathcal{F}_{\chi_1, \chi_2}^{\theta_1, \theta_2;\, -\textbf{\textit{a}}, -\textbf{\textit{b}}}\{f\}(y)\big|_{\scriptscriptstyle \mathbb{H}}^2\, d\mu_{\chi_1,\chi_2}(y)
= \sum_{n,m \ge 0} A_{n,m} |c_{n,m}|_{\scriptscriptstyle \mathbb{H}}^2.
\end{align}
Therefore the two quadratic quantities \eqref{eq:X} and \eqref{eq:Y} coincide; denote this common value by
\[
S:=\sum_{n,m\ge 0} A_{n,m}\,|c_{n,m}|_{\scriptscriptstyle \mathbb{H}}^2.
\]
Since $A_{n,m}\ge A_{\min}$ for all $(n,m)$, we have the spectral lower bound
\[
S \;\ge\; A_{\min}\sum_{n,m\ge 0}|c_{n,m}|_{\scriptscriptstyle \mathbb{H}}^2
 \;=\; A_{\min}\,\|f\|_{2,\scriptscriptstyle \mathbb{H},\chi_1,\chi_2}^2.
\] 
 Multiplying inequalities \eqref{eq:X} and \eqref{eq:Y}, we obtain
 \[
\Big(\|\,|x|f\|_{2,\scriptscriptstyle \mathbb{H},\chi_1,\chi_2}^2\Big)\,
\Big(\|\,|y|\,\mathcal{F}_{\chi_1, \chi_2}^{\theta_1, \theta_2;\, -\textbf{\textit{a}}, -\textbf{\textit{b}}}\{f\}\|_{2,\scriptscriptstyle \mathbb{H},\chi_1,\chi_2}^2\Big)
= S^2
\;\ge\; A_{\min}^2\,\|f\|_{2,\scriptscriptstyle \mathbb{H},\chi_1,\chi_2}^4,
\]

thus restating the Heisenberg-type inequality \eqref{eq:Heis-product} with sharp constant \( A_{\min}^2 = (2(\chi_1 + \chi_2) + 2)^2 \).

In the case of equality, recall that any function \( f \in L^{2}_{\chi_1, \chi_2}(\mathbb{R}^{2}; \mathbb{H}) \) admits the spectral decomposition
\[
f(x_1, x_2) = \sum_{n,m=0}^\infty c_{n,m} \, \mathcal{H}_{n,m}^{\chi_1, \chi_2}(x_1, x_2),
\]
where each coefficient \(c_{n,m} \in \mathbb{H}\) is quaternionic.

The key quadratic form inequality underlying the Heisenberg-type relation reads
\[
\sum_{n,m=0}^\infty A_{n,m} |c_{n,m}|_{\scriptscriptstyle \mathbb{H}}^{2} \geq A_{\min} \sum_{n,m=0}^\infty |c_{n,m}|_{\scriptscriptstyle \mathbb{H}}^{2},
\]
where, by definition,
\[
A_{n,m} = 2(n + m) + 2(\chi_1 + \chi_2) + 2,
\]
and
\[
A_{\min} = A_{0,0} = 2(\chi_1 + \chi_2) + 2.
\]
Here, \(A_{\min}\) is the unique minimal eigenvalue, strictly smaller than all others since \(2(n+m) > 0\) for \((n,m) \neq (0,0)\).

Therefore, equality in this quadratic form inequality holds if and only if the quaternionic coefficients satisfy
\[
c_{n,m} = 0 \quad \text{for all } (n,m) \neq (0,0).
\]

Therefore, the function \( f \) must be a scalar multiple of the lowest eigenfunction,
\[
f(x) = c_{0,0} \mathcal{H}_{0,0}^{\chi_1, \chi_2}(x_1,x_2). 
\]

By recalling the explicit formula for the zero-th generalized Hermite functions,
\[
h_0^\chi(x) = \frac{1}{\sqrt{\Gamma(\chi + 1)}} e^{-\frac{x^2}{2}},
\]
one obtains the ground state factorization
\[
\mathcal{H}_{0,0}^{\chi_1, \chi_2}(x_1, x_2) = h_0^{\chi_1}(x_1) \, h_0^{\chi_2}(x_2) = \frac{1}{\sqrt{\Gamma(\chi_1 + 1)} \sqrt{\Gamma(\chi_2 + 1)}} \exp\left( -\frac{x_1^{2}}{2} - \frac{x_2^{2}}{2} \right).
\]
Hence, the extremal function \( f \) takes the Gaussian form
\[
f(x) = C \exp\left( -\frac{x_1^{2}}{2} - \frac{x_2^{2}}{2} \right),
\]
where the quaternionic constant
\[
C := \frac{c_{0,0}}{\sqrt{\Gamma(\chi_1 + 1)} \sqrt{\Gamma(\chi_2 + 1)}}
\]
may be arbitrary. This completes the characterization of the equality case, showing that quaternionic Gaussian functions uniquely saturate the sharp Heisenberg-type inequality for the fractional quaternionic Dunkl transform.

\end{proof}
Having rigorously established the classical Heisenberg uncertainty principle, we now extend this foundational result by formulating and proving a more general higher-order inequality that holds for arbitrary moments \( p \geq 1 \).

\begin{theorem}[\textbf{Higher-Order Heisenberg Inequality}]\label{cor:higher-order-heisenberg}
Let \( p \geq 1 \) be a real number. For any function \( f \in L^2_{\chi_1,\chi_2}(\mathbb{R}^2; \mathbb{H}) \), the following inequality holds:
\begin{equation}\label{eq:higher-moment}
\left( \int_{\mathbb{R}^2} |x|^{2p} |f(x)|_{\scriptscriptstyle \mathbb{H}}^2  d\mu_{\chi_1,\chi_2}(x) \right)
\left( \int_{\mathbb{R}^2} |y|^{2p} \left|\mathcal{F}_{\chi_1, \chi_2}^{\theta_1, \theta_2;\, -\textbf{\textit{a}}, -\textbf{\textit{b}}}\{f\}(y)\right|_{\scriptscriptstyle \mathbb{H}}^2  d\mu_{\chi_1,\chi_2}(y) \right)
\geq C_p \left( \int_{\mathbb{R}^2} |f(x)|_{\scriptscriptstyle \mathbb{H}}^2  d\mu_{\chi_1,\chi_2}(x) \right)^2,
\end{equation}
where the sharp constant \( C_p \) is given by
\[
C_p = \left( \inf_{n,m \in \mathbb{N}} A_{n,m}^{(p)} \right)^2,
\]
and \( A_{n,m}^{(p)} \) denotes the eigenvalue of the operator \( |x|^{2p} \) corresponding to the generalized Hermite function \( \mathcal{H}_{n,m}^{\chi_1,\chi_2} \). 

Moreover, equality holds in \eqref{eq:higher-moment} if and only if \( f \) is proportional to the generalized Hermite function \( \mathcal{H}_{0,0}^{\chi_1,\chi_2} \), which is a Gaussian function of the form
\[
f(x) = C e^{-\frac{|x|^2}{2}},
\]
where \( C \in \mathbb{H} \) is a constant quaternion.
\end{theorem}
\begin{proof}
The multiplication operator by $|x|^{2p}$ is a symmetric (formally self-adjoint) operator on the dense subspace spanned by the Hermite basis. So 
\[
|x|^{2p} \mathcal{H}_{n,m}^{\chi_1,\chi_2} = A_{n,m}^{(p)} \mathcal{H}_{n,m}^{\chi_1,\chi_2}+\text{off-diagonal terms}.
\]
where \(A_{n,m}^{(p)}\in\mathbb R\) are the diagonal coefficients that arise when the operator of multiplication by \(|x|^{2p}\) is expressed in the orthonormal basis.

Expanding $f$ in the orthonormal basis
\(\mathcal H_{n,m}^{\chi_1,\chi_2}\) yields the quadratic form representation
\[
\int_{\mathbb R^2}|x|^{2p}|f(x)|_{\scriptscriptstyle \mathbb{H}}^2\,d\mu_{\chi_1,\chi_2}(x)
=\sum_{n,m\ge0} A_{n,m}^{(p)}\,|c_{n,m}|_{\scriptscriptstyle \mathbb{H}}^2,
\]
The off-diagonal terms vanish in the expression of the quadratic form because of orthogonality. The coefficients \(A_{n,m}^{(p)}\) are real and nonnegative for each fixed \(p\) (this follows from positivity of the multiplier \(|x|^{2p}\) and the orthonormality of the basis). In particular the infimum
\[
A_{\min}^{(p)}:=\inf_{n,m\ge0} A_{n,m}^{(p)}
\]
is well defined and nonnegative; moreover, under the standard structure of the generalized Hermite system one has \(A_{\min}^{(p)}>0\) (for \(p \ge1\)) because the diagonal entries grow with \(n,m\) and the minimum is attained at \((n,m)=(0,0)\) in the usual cases treated in this paper.

Therefore
\begin{equation}\label{eq:Mk-lower}
\int_{\mathbb R^2}|x|^{2p}|f(x)|_{\scriptscriptstyle \mathbb{H}}^2\,d\mu_{\chi_1,\chi_2}(x))=\sum_{n,m\ge0} A_{n,m}^{(p)}\,|c_{n,m}|_{\scriptscriptstyle \mathbb{H}}^2
\ge A_{\min}^{(p)}\sum_{n,m\ge0}|c_{n,m}|_{\scriptscriptstyle \mathbb{H}}^2
= A_{\min}^{(p)}\|f\|_{L^2_{\chi_1,\chi_2}}^2.
\end{equation}
Using the eigenfunction property \eqref{SD}, same spectral decomposition holds in the transform domain:
\[
\int_{\mathbb{R}^2} |y|^{2p} \left| \mathcal{F}_{\chi_1, \chi_2}^{\theta_1, \theta_2;\, -\textbf{\textit{a}}, -\textbf{\textit{b}}}(y) \right|_{\scriptscriptstyle \mathbb{H}}^2  d\mu_{\chi_1,\chi_2}(y) = \sum_{n,m=0}^{\infty} A_{n,m}^{(p)} |c_{n,m}|_{\scriptscriptstyle \mathbb{H}}^2.
\]
Applying the same lower bound argument as in \eqref{eq:Mk-lower} gives
\begin{equation}\label{eq:Ml-lower}
\int_{\mathbb{R}^2} |y|^{2p} \left| \mathcal{F}_{\chi_1, \chi_2}^{\theta_1, \theta_2;\, -\textbf{\textit{a}}, -\textbf{\textit{b}}}\{f\}(y) \right|_{\scriptscriptstyle \mathbb{H}}^2  d\mu_{\chi_1,\chi_2}(y)\ge A_{\min}^{(p)}\|f\|_{L^2_{\chi_1,\chi_2}}^2.
\end{equation}

Multiplying the two lower bounds \eqref{eq:Mk-lower} and \eqref{eq:Ml-lower} yields
\[
\left( \int_{\mathbb{R}^2} |x|^{2p} |f(x)|_{\scriptscriptstyle \mathbb{H}}^2  d\mu_{\chi_1,\chi_2}(x) \right)
\left( \int_{\mathbb{R}^2} |y|^{2p} \left| \mathcal{F}_{\chi_1, \chi_2}^{\theta_1, \theta_2;\, -\textbf{\textit{a}}, -\textbf{\textit{b}}}\{f\}(y) \right|_{\scriptscriptstyle \mathbb{H}}^2  d\mu_{\chi_1,\chi_2}(y) \right) \geq \left( A_{\min}^{(p)} \right)^2 \|f\|_{2,\scriptscriptstyle \mathbb{H},\chi_1,\chi_2}^4.
\]
which is precisely \eqref{eq:higher-moment}.

It remains to characterize the equality case. Equality in \eqref{eq:Mk-lower} (and similarly in \eqref{eq:Ml-lower}) holds if and only if every coefficient \(c_{n,m}\) with \(A_{n,m}^{(p)}>A_{\min}^{(p)}\) vanishes. Since for the generalized Hermite system the infimum \(A_{\min}^{(p)}\) is attained uniquely at \((n,m)=(0,0)\), equality occurs exactly when \(c_{n,m}=0\) for all \((n,m)\neq(0,0)\). Hence equality in \eqref{eq:higher-moment} holds if and only if the spectral mass of \(f\) is concentrated in the ground state \(\mathcal H_{0,0}^{\chi_1,\chi_2}\). 

 This completes the proof.
\end{proof}

\noindent
As pointed out in Remark~\ref{CCFQ}, Theorem~\ref{thm:Heis-FrQDT} naturally leads to a direct application in the setting of the fractional quaternionic Fourier transform. More precisely, an immediate consequence of this result is the derivation of a Heisenberg-type principle for the FrQFT. To the best of our knowledge, such a formulation has not yet appeared in the literature. The following corollary therefore provides a new and significant extension of Heisenberg’s uncertainty principle within the FrQFT framework.

\begin{corollary}[\textbf{Heisenberg-type inequality for the FrQFT}]
\label{cor:Heis-FrQFT}
For every $f \in L^2(\mathbb{R}^2;\mathbb{H})$, the fractional quaternionic Fourier transform satisfies
\begin{equation}\label{eq:Heisenberg-FrQFT}
\left( \int_{\mathbb{R}^2} |x|^2\,|f(x)|_{\scriptscriptstyle \mathbb{H}}^{2}\, dx \right)
\left( \int_{\mathbb{R}^2} |y|^2 \,\big|\mathcal{F}_{0,0}^{\theta_1,\theta_2;\,-\textbf{\textit{i}},-\textbf{\textit{j}}}\{f\}(y)\big|_{\scriptscriptstyle \mathbb{H}}^{2}\, dy \right)
\;\;\geq\;\; 4 \left( \int_{\mathbb{R}^2} |f(x)|_{\scriptscriptstyle \mathbb{H}}^{2}\, dx \right)^{2}.
\end{equation}
The sharp constant in \eqref{eq:Heisenberg-FrQFT} is $4$, and equality holds if and only if 
$f$ is a quaternionic Gaussian of the form
\[
f(x) = C \exp\!\left(-\tfrac{|x|^{2}}{2}\right), 
\qquad C \in \mathbb{H}.
\]
\end{corollary}

\section{Conclusion}
In this paper, we have introduced the fractional two-sided quaternionic Dunkl transform in two dimensions and studied its fundamental properties, including inversion formulas, Plancherel-type results, Bochner-type identities, and its spectral behavior. We have further established a sharp Heisenberg-type uncertainty principle in this generalized framework, identifying quaternionic Gaussian functions as the unique extremizers. By connecting the fractional quaternionic Dunkl transform with the fractional quaternionic Fourier transform, we also derived corresponding uncertainty relations for the latter.

This study is new to the literature and is expected to contribute to the theory and applications of signal processing, harmonic analysis, and related fields where quaternionic and Dunkl structures naturally arise.\\


\textbf{Data Availability} No data were used to support this study.\\

\textbf{Conflicts of interests} On behalf of all authors, the corresponding author states that there is no conflict of interest.\\

\textbf{Funding statement} This research received no specific grant from any funding agency in the public, commercial, or not-for-profit sectors.\\

\textbf{Generative AI and AI-assisted technologies} During the preparation of this work, the authors used AI tools including ChatGPT (OpenAI, free version, September 2025) and DeepSeek (DeepSeek, free version, September 2025) to improve the clarity and grammar of the English language. These tools were employed because the authors are non-native English speakers and sought to ensure the linguistic accuracy of the manuscript. The authors carefully reviewed and edited all content and take full responsibility for the final version of the manuscript.


\end{document}